\documentclass[final,onefignum,onetabnum]{siamart190516}
\usepackage{amsfonts}
\usepackage{graphicx}
\usepackage{subcaption}
\usepackage{epstopdf}
\usepackage{algorithmic}
\usepackage{pythonhighlight}
\usepackage{float}
\ifpdf
  \DeclareGraphicsExtensions{.eps,.pdf,.png,.jpg}
\else
  \DeclareGraphicsExtensions{.eps}
\fi

\newsiamremark{remark}{Remark}
\newsiamremark{hypothesis}{Hypothesis}
\crefname{hypothesis}{Hypothesis}{Hypotheses}
\newsiamthm{claim}{Claim}

\headers{A Generalization of QR Factorization to Non-Euclidean Norms}{R. Atcheson}

\title{A Generalization of QR Factorization to Non-Euclidean Norms}

\author{Reid Atcheson\thanks{Numerical Algorithms Group Inc.
  (\email{reid.atcheson@nag.com}, \url{https://www.reidatcheson.com/}).}}

\usepackage{amsopn}

\DeclareMathOperator*{\argmin}{arg\,min}
\DeclareMathOperator*{\linearspace}{span}
\newcommand{\norm}[1]{\left\lVert#1\right\rVert}

\ifpdf
\hypersetup{
  pdftitle={A Generalization of QR Factorization To Non-Euclidean Norms},
  pdfauthor={R. Atcheson}
}
\fi

\begin{document}
\maketitle

\begin{abstract}
  I propose a way to use non-Euclidean norms to formulate a QR-like factorization which can unlock interesting 
  and potentially useful properties of non-Euclidean norms - for example the ability of $l^1$ norm to suppresss outliers or promote sparsity.
  A classic QR factorization of a matrix $\mathbf{A}$ 
  computes an upper triangular matrix $\mathbf{R}$ and orthogonal matrix $\mathbf{Q}$ such that $\mathbf{A} = \mathbf{QR}$. To generalize
  this factorization to a non-Euclidean norm $\| \cdot \|$
  I relax the orthogonality requirement for $\mathbf{Q}$ and instead require it have condition number 
  $\kappa \left ( \mathbf{Q} \right )  = \| \mathbf{Q} ^{-1} \| \| \mathbf{Q} \|$ that is bounded independently of $\mathbf{A}$. I present the algorithm for computing $\mathbf{Q}$ and
  $\mathbf{R}$ and prove that this algorithm results in $\mathbf{Q}$ with the desired properties. I also prove that this algorithm generalizes
  classic QR factorization in the sense that when the norm is chosen to be Euclidean: $\| \cdot \|=\| \cdot \|_2$ then $\mathbf{Q}$ is orthogonal. Finally
  I present numerical results confirming mathematical results with $l^1$ and $l^{\infty}$ norms. I supply Python code for experimentation.
\end{abstract}


\section{Introduction}
\label{sec:introduction}
The QR factorization allows for highly stable, robust, and efficient matrix factorization with
beneficial properties to many areas of numerical linear algebra. The factorization may be 
implemented using only stable operations such as Householder reflectors or Givens rotations \cite{high:ASNA2} and
highly efficient blocked implementations can use level-3 BLAS \cite{golub2013matrix},\cite{lapack}. The key property
of the factorization is orthogonality. Orthgonality means that if we QR factorize a matrix $\mathbf{A}=\mathbf{Q}\mathbf{R}$ then
$\mathbf{Q}^T = \mathbf{Q}^{-1}$. This property has enormous utility in numerical linear algebra and it plays a significant role
in almost every eigenvalue algorithm \cite{templateseig} as well as least-squares solvers \cite{golub2013matrix}. Despite huge
success of the QR factorization it has resisted generalization in large part because the orthgonality property tightly bound 
to the underlying Euclidean norm $\norm{\cdot}_2$ and is almost meaningless without it. Indeed a well-known fact is that
the constraint of orthogonality on $\mathbf{Q}$ almost completely determines the result of the algorithm, leaving little room
for alterations that could benefit a new domain - for example by using a norm with domain-specific advantages over the Euclidean norm.

Non-Euclidean norms can sometimes provide domain-specific benefits. It is now well known for example that the $l^1$ norm,
when used for regression (where it is called "Least Absolut Deviations") is far less sensitive outliers than the $l^2$ norm \cite{birkes2011alternative},
and the $l^1$ norm also can promote sparsity compared to the $l^2$ norm \cite{Donoho2006ForML}. The $l^{\infty}$ norm also has utility
for minimax problems \cite{minimax}. These properties of some non-Euclidean norms have recently resulted in new matrix factorizations
using these norms as opposed to the $l^2$ norm, for example $l^1$ norm based SVD factorizations \cite{ke2005robust}. I determined to 
investigate whether a we can similarly modify a QR factorization to inherit benefits from non-Euclidean norms.

As already mentioned the orthogonality property of the QR factorization nearly completely determines it and it simultaneously
locks us in to the $l^2$ norm, thus we need a way to relax this condition but find an analogous condition
which provides similar utility. For this work I focus on the \emph{conditioning} of $\mathbf{Q}$ rather than orthogonality.
An orthogonal matrix has condition number (in the $l^2$ norm) equal to $1$. Thus I present algorithm \ref{alg:genqr} which
for any prescribed norm can produce a QR-like factorization where the resulting matrix $\mathbf{Q}$ 
is well-conditioned in the supplied norm. One key contribution of this work 
is the statement and proof of this key conditioning theorem \ref{thm:inverse} which
bounds the norm of the inverse of $\mathbf{Q}$. I also show in theorem \ref{thm:classicqr} that this algorithm
generalizes the classic QR factorization in the sense that if the prescribed norm
is the Euclidean norm then $\mathbf{Q}$ is orthogonal. I follow the mathematical proofs with numerical experiments using the $l^1$ and $l^{\infty}$ norms.

\section{Main results}

The main results of this work are mathematical with some light numerical experiments for illustration purposes.
In this section I first present the algorithm \ref{alg:genqr} below. I then state and prove key bounds
on the resulting matrix $\mathbf{Q}.$ I first state and prove the forward bound \ref{thm:forward} which
shows that while $\mathbf{Q}$ does not have orthogonality, it still does not increase the norm of vectors
significantly when applied to them (an orthogonal matrix by comparison does not increase the norm of a vector
at all). The next theorem \ref{thm:inverse} shows a similar kind of bound but in the other direction: how much
can it \emph{shrink} an input vector. As usual an orthogonal matrix will not shrink an input vector at all,
but since the algorithm \ref{alg:genqr} does not guarantee orthogonality I instead provide bounds that
constrain its conditioning.

I freely use the following vector norms throughout:
\begin{definition}
  Suppose that $m$ is a positive integer and that $\mathbf{x}\in\mathbb{R}^m$. Define the following:
  \begin{align}
    \norm{x}_1        & = \sum_{j=1}^m \left | x_j \right | & \text{   } l^1-\text{norm} \\
    \norm{x}_2        & = \sqrt{\sum_{j=1}^m x_j^2}         & \text{   } l^2-\text{norm} \\
    \norm{x}_{\infty} & = \max_j \left | x_j \right |     & \text{   } l^{\infty}-\text{norm}
  \end{align}
\end{definition}

The core algorithm under investigation follows.For simplicity of presentation I focus
on the case where the input matrix $\mathbf{A}$ is full-rank and square, but I also show that
the algorithm and subsequent theorems may be trivially extended to low-rank or rectangular cases
in section \ref{sec:lowrank}.

\begin{algorithm}[H] 
\caption{Generalized QR Factorization}
\label{alg:genqr}

  Start with an input $A\in\mathbb{R}^{m\times m}$ and any norm 
  $\| \cdot \|$ on $ \mathbb{R} ^ m $.

  I use

  \begin{align}
    A &= (A_1,A_2,A_3,\ldots,A_m) \\
    Q &= (Q_1,Q_2,Q_3,\ldots,Q_m) 
  \end{align}

  to represent $A$ and $Q$ by their respective columns. Furthermore I define 
  $A^i,Q^i \in \mathbb{R} ^ {m \times i}$ as the first $i$ columns of $A,Q$ respectively:

  \begin{displaymath}
    Q^i = (Q_1,Q_2,\ldots,Q_i).
    A^i = (A_1,A_2,\ldots,A_i).
  \end{displaymath}

  I now define the $Q$ and $R$ factors inductively as follows:

  \begin{align} 
    \label{eq:recurrencebase1}
    Q_1 &= A_1 / \| A_1 \| \\
    \label{eq:recurrencebase2}
    R(1,1) &= \frac{1}{\| A_1 \|}
  \end{align}

  and for any $1\leq j \leq m-1$ I define

  \begin{align}
    \label{eq:recurrence1}
    c^j &=\argmin _ {c^j \in \mathbf{R}^j }\|  A_{j} - Q^j c^j \|  \\
    \label{eq:recurrence2}
    \gamma_j &= \| A_{j} - Q^jc^j \| \\
    \label{eq:recurrence3}
    Q^{j+1} &= (Q^j, \gamma_j ^{-1} ( A_{j} - Q^j c^j) ) \\
    \label{eq:recurrence4}
    R(j,1:j-1) &= c^j \\
    \label{eq:recurrence5}
    R(j,j) &= \gamma_j 
  \end{align}

\end{algorithm}

\begin{theorem}[Generalized QR factorization]\label{thm:qrfact}
  Suppose that $m$ is a positive integer, that $\mathbf{A} \in \mathbb{R}^{m \times m}$ is full-rank, that $\|\cdot\|$ is a norm, and that $\mathbf{Q},\mathbf{R}$ are output
  from algorithm \ref{alg:genqr}. Then
  \begin{displaymath} \label{eq:qrequation}
    \mathbf{A} = \mathbf{QR}
  \end{displaymath}
\end{theorem}

\begin{proof}
  I proceed by mathematical deduction. Note that $A_1 = Q^1R(1,1)$ follows directly
  from the base case definitions of these quantities \ref{eq:recurrencebase1},\ref{eq:recurrencebase2}. Now assume $A^j = Q^jR(1:j,1:j)$  
  for some $j>1.$ Then from \ref{eq:recurrence1},\ref{eq:recurrence2},\ref{eq:recurrence3},\ref{eq:recurrence4},\ref{eq:recurrence5} we have

  \begin{align*}
    Q^{j+1}R(1:j+1,1:j+1) &= (Q^j,Q_{j+1})
    \begin{bmatrix}
       R(1:j,1:j) & c_j^k \\
      0              & \gamma_j
    \end{bmatrix} \\
    &= (Q^j R(1:j,1:j),Q_{j}c^j + \gamma_j^{-1} Q_{j+1}) \\
    &= (Q^j R(1:j,1:j),Q_{j}c_j^k + \gamma_j^{-1}\gamma_j \left ( A_j - Q^j c^j \right ) \\
    &= (A^j,A_{j+1}) \\
    &= A^{j+1}
  \end{align*}

  establishing the equation

  \begin{align*}
    A^j = Q^j R(1:j,1:j)
  \end{align*}

  for all nonnegative integers $j\leq m.$ Taking $j=m$ proves \ref{eq:qrequation}

\end{proof}
The first theorem related to conditioning of $\mathbf{Q}$ establishes a simple forward bound on its norm.
\begin{theorem}[Forward bounds on Q]\label{thm:forward}
  Suppose that $m$ is a positive integer and that $\|\cdot\|$ is a norm on $\mathbb{R}^m$.
  Then there exists $C_1>0$ such that for every full-rank $\mathbf{A} \in \mathbb{R}^{m \times m}$ and every $\mathbf{x}\in\mathbb{R}^m$ we have

  \begin{displaymath}
    \| \mathbf{Q}\mathbf{x} \| \leq C_1 \|\mathbf{x}\|
  \end{displaymath}

  where $\mathbf{Q}$ is output from algorithm \ref{alg:genqr}
\end{theorem}

\begin{proof}
  Suppose that $x\in\mathbb{R}^m$.
  \begin{align*}
    \|Qx\| &=    \left \| \sum _{i=1}^m Q_i x_i \right \| \\
           &\leq  \sum_{i=1}^m \| Q_i x_i \| \\
           &\leq  \sum_{i=1}^m \| Q_i\| \left | x_i \right | \\
           &=  \| x \| _1 
  \end{align*}
  Finally we may apply norm
  equivalence between all norms in finite dimensional spaces to choose $C_1>0$ such that
  $\| x \| _1 \leq C_1 \|x \|$ holds for all $x$ 
\end{proof}
\begin{remark}
  The constant $C_1$ produced above is independent of $\mathbf{A}$ but still 
  (likely) depends on the dimension $m$ of the space because 
  of the use of norm equivalence.
\end{remark}

\begin{theorem}[Inverse bounds on Q]\label{thm:inverse}
  Suppose that $m$ is a positive integer and that $\|\cdot\|$ is a norm on $\mathbb{R}^m$.
  Then there exists $C_2>0$ such that for every full-rank $\mathbf{A} \in \mathbb{R}^{m \times m}$ and every $\mathbf{x}\in\mathbb{R}^m$ we have
  \begin{displaymath}
    C_2 \| \mathbf{x} \| \leq \| \mathbf{Q}\mathbf{x} \|
  \end{displaymath}
  where $\mathbf{Q}$ is output from algorithm \ref{alg:genqr}
\end{theorem}
Theorem \ref{thm:inverse} is a little more involved than those that preceeded it, so I organize its proof into a sequence of lemmas followed by
main proof. The lemmas effectively establish partial bounds which if combined carefully result in the complete inverse bound.
The first lemma establishes a partial inverse bound on $\mathbf{Q}$ using the optimality properties of its columns.
\begin{lemma}[An optimality property of $\mathbf{Q}$]\label{lem:optim}
  Suppose that $m$ is a positive integer, that $\| \cdot \|$ is a norm on $\mathbb{R}^m$, and that $\mathbb{A}\in\mathbb{R}^{m\times m}$.
  Suppose further that $k$ is such that $0\leq k \leq m$ and that $x_1,\ldots,x_k$ are real numbers such that $x_k \neq 0$. Then
  \begin{displaymath}
    \left \| \sum_{j=1}^k Q_k x_k \right \| \geq |x_k|
  \end{displaymath}
  where $\mathbf{Q}$ is produced by algorithm \ref{alg:genqr}
\end{lemma}
\begin{proof}[Proof of lemma \ref{lem:optim}]
  From the inductive definition \ref{eq:recurrence3} of $\mathbf{Q}$ we have
  \begin{align*}
    \norm{\sum_{j=1}^k Q_j x_j } &= \norm{\sum_{j=1}^{k-1} Q_j x_j + Q_kx_k} \\
                                 &= \norm{\sum_{j=1}^{k-1} Q_j x_j +  \gamma_k^{-1}\left( A_k - \sum_{j=1}^{k-1}Q_jc_j^k \right)x_k } \\
  \end{align*}
  where $c^k$ are coefficients which solve the minimization problem \ref{eq:recurrence1}. We may rearrange terms as follows
  \begin{align*}
    &\norm{\sum_{j=1}^{k-1} Q_j x_j +  \gamma_k^{-1}\left( A_k - \sum_{j=1}^{k-1}Q_jc_j^k \right)x_k }  \\
    &= \norm{ \sum_{j=1}^{k-1}Q_j (x_j - \gamma_k^{-1}c_j^kx_k)+\gamma _k ^{-1} A_n x_k}\\
    &= \left | \gamma _k ^{-1} x_k \right |  \norm{ \gamma_k x_k^{-1} \sum_{j=1}^{k-1}Q_j (x_j - \gamma_k^{-1}c_j^kx_k)+ A_k }\\
    &\geq \left | \gamma _k ^{-1} x_k \right | \norm{ \sum_{j=1}^{k-1}Q_jc_j^k - A_j} \\
    &= \left | x_k \right |
  \end{align*}
  the inequality follows from optimality properties of $c^k$ as specified in \ref{eq:recurrence1} and the final equation
  from the definition of $\gamma _k $ in \ref{eq:recurrence2}.
\end{proof}
The next lemma establishes a partial bound for $\mathbf{Q}$ when the vector $\mathbf{x}$ satisfies a decay property.
\begin{lemma}[An inequality dependent on certain decay property]\label{lem:decay}
  Suppose that $m$ is a positive integer, that $\| \cdot \|$ is a norm on $\mathbb{R}^m$, and that $\mathbb{A}\in\mathbb{R}^{m\times m}$.
  Suppose further that $k$ is such that $0\leq k \leq m$ and that $\mathbf{x}\in \mathbb{R}^m$ satisfies the following decay property:
  \begin{displaymath}
    \left | x_j \right | \leq \frac{1}{2^j}\norm{\mathbf{x}}_{\infty} \text{   } (j=k,\ldots,m)
  \end{displaymath}
  Then we have
  \begin{displaymath}
    \norm{\sum_{j=k}^m Q_j x_j } \leq \sum_{j=k}^m \frac{1}{2^k} \norm{\mathbf{x}}_{\infty}
  \end{displaymath}
  where $\mathbf{Q}$ is produced by algorithm \ref{alg:genqr}
\end{lemma}
\begin{proof}[Proof of lemma \ref{lem:decay}]
  This follows by application of triangle inequality, recognizing that $\norm{Q_j}=1$ for all $j$, and then applying the decay property of $\mathbf{x}$
\end{proof}

Now I prove the main fact below.
\begin{proof}[Proof of theorem \ref{thm:inverse}]
  For this proof I use the $l^{\infty}$ norm defined as:
  \begin{displaymath}
    \norm{\mathbf{y}}_{\infty} = \max_{j} \left | y_j \right |
  \end{displaymath}
  Suppose that $\mathbf{A}\in\mathbb{R}^{m\times m}$ and that $\mathbf{x}\in\mathbb{R}^m.$ Define the integer $k$  to be the largest integer 
  that satisfies $0\leq k \leq m$ and the following inequality:
  \begin{displaymath}
    |x_k| \geq \frac{1}{2^k} \norm{\mathbf{x}}_{\infty}
  \end{displaymath}
  By definition of $k$ we see that $\mathbf{x}$ satisfies the decay property stated in lemma \ref{lem:decay} for $x_{k+1},\ldots,x_{m}$. Before
  proceeding with the key inequality I first establish nonnegativity of a key term so that we may later remove absolute values from it:
  \begin{align*}
     \norm{\sum_{j=k+1}^m Q_j x_j} &\leq \sum_{j=k+1}^m \frac{1}{2^j} \norm{\mathbf{x}}_{\infty} \\
                                   &\leq \frac{1}{2^k}\norm{\mathbf{x}}_{\infty} \\
                                   &\leq \left | x_k \right | \\
                                   &\leq \norm{\sum_{j=1}^k Q_j x_j}  
  \end{align*}
  Thus we have
  \begin{align*}
  \left | \norm{\sum_{j=1}^k Q_j x_j} - \norm{\sum_{j=k+1}^m Q_j x_j} \right | &= \norm{\sum_{j=1}^k Q_j x_j} - \norm{\sum_{j=k+1}^m Q_j x_j} \\
                                                                               &\geq 0
  \end{align*}

  Thus by applying
  the reverse triangle inequality, lemma \ref{lem:decay}, and then lemma \ref{lem:optim}, we find
  \begin{align*}
    \norm{\mathbf{Q}\mathbf{x}} &= \norm{\sum_{j=1}^m Q_j x_j}  \\
                                &= \norm{\sum_{j=1}^k Q_j x_j + \sum_{j=k+1}^m Q_j x_j} \\
                                &\geq  \left | \norm{\sum_{j=1}^k Q_j x_j}  - \norm{\sum_{j=k+1}^m Q_j x_j} \right | \\
                                &=  \norm{\sum_{j=1}^k Q_j x_j}  - \norm{\sum_{j=k+1}^m Q_j x_j} \\
                                &\geq  \norm{\sum_{j=1}^k Q_j x_j}  -  \sum_{j=k+1}^m \frac{1}{2^j}\norm{x}_{\infty} \\
                                &\geq  \left | x_k \right |  -  \sum_{j=k+1}^m \frac{1}{2^j}\norm{\mathbf{x}}_{\infty} \\
                                &\geq  \frac{1}{2^k}\norm{\mathbf{x}}_{\infty}   -  \sum_{j=k+1}^m \frac{1}{2^j}\norm{\mathbf{x}}_{\infty} \\
                                &= (\frac{1}{2^k} - \sum_{j=k+1}^m \frac{1}{2^j})\norm{\mathbf{x}}_{\infty} \\
                                &\geq \frac{1}{2^m}\norm{\mathbf{x}}_{\infty} 
  \end{align*}
  Next we may apply norm equivalence in finite dimensional spaces to choose a constant $K>0$ such that
  \begin{displaymath}
    \norm{\mathbf{y}} \leq  K \norm{\mathbf{y}}_{\infty}
  \end{displaymath}
  holds for all $\mathbf{y} \in \mathbb{R}^m$. Finally by defining $C_2 = K\frac{1}{2^m}$ we see that
  \begin{displaymath}
    \norm{\mathbf{Q}\mathbf{x}} \geq C_2 \norm{\mathbf{x}}
  \end{displaymath}
  completing the proof
\end{proof}

\begin{corollary}[Condition number bounds for Q]\label{thm:condition}
  Suppose that $m$ is a positive integer and that $\|\cdot\|$ is a norm on $\mathbb{R}^m$.
  Then there exists $C>0$ such that for every full-rank $\mathbf{A} \in \mathbb{R}^{m \times m}$ and every $\mathbf{x}\in\mathbb{R}^m$ we have
  \begin{displaymath}
    \kappa \left (\mathbf{Q} \right) = \| \mathbf{Q}^{-1} \| \| \mathbf{Q} \| \leq C
  \end{displaymath}
  where $\mathbf{Q}$ is output from algorithm \ref{alg:genqr}
\end{corollary}

\begin{proof}
  Applying theorems \ref{thm:forward} and \ref{thm:inverse} together we may take $C = \frac{C_1}{C_2}$ and see immediately that it establishes the desired bound.
\end{proof}

\begin{remark}
  Note that while $C_1$ \emph{likely} depends on the dimension $m$ of the space because of the use of norm equivalence, $C_2$ \emph{almost certainly} depends on $m$.
  The best bound achieved here has $C_2$ decaying exponentially with $m$, leading to an exponentially growing condition number of $Q$ if the bound is sharp. We will
  see however in numerical experiments presented in section \ref{sec:experiments} that this bound appears to be a much more manageable $O(\frac{1}{m})$
  for the $l^1$ and $l^{\infty}$ norms.

  The key of this theorem wasn't necessarily a \emph{provably small} bound, but rather that, regardless of the norm $\| \cdot \|$ 
  the bound is independent of $\mathbf{A},$ so in particular $\mathbf{A}$ may be nearly numerically singular and $\mathbf{Q}$ still
  mathemtatically has the same conditioning. It may be possible if we restrict ourselves to specific norms to prove much more lenient
  bounds. Of course we already know that for the $l^2$ norm we have $C_1 = C_2 = 1$ (see theorem \ref{thm:classicqr} below).
\end{remark}
Finally I show that when we take the input norm as the classic Euclidean norm then the factorization becomes a classic QR factorization.
\begin{theorem}[Classic QR as special case]\label{thm:classicqr}
  Suppose that $m$ is a positive integer, that $\mathbf{A} \in \mathbb{R}^{m \times m}$ is full-rank, that $\| \cdot \| = \| \cdot \| _2 $ is the classic Euclidean norm, 
  and that $\mathbf{Q},\mathbf{R}$ are output from algorithm \ref{alg:genqr}. Then
  \begin{displaymath}
    \mathbf{Q}^{-1} = \mathbf{Q}^T
  \end{displaymath}
\end{theorem}

\begin{proof}
  By the inductive definition of $Q$ in \ref{eq:recurrence3} we have
  \begin{equation}
    Q^{j+1} = (Q^j, \gamma_j ^{-1} ( A_j - Q^j c^j) )
  \end{equation}
  Recall that $c^j$ solves the minimization problem
  \begin{equation}
    c^j =\argmin _ {c^j \in \mathbf{R}^j } \norm{  A_j - Q^j c^j }_2  
  \end{equation}
  which means it is forming the $l^2$ projection of $A_k$ onto the space $V=\linearspace(Q_1,\ldots,Q_j)$.
  Since $Q_{j+1}$ is the residual of this projection, it is orthogonal to the whole space $V$. 

  In other words the above shows that the columns of $\mathbf{Q}$ are mutually orthogonal, and the columns are also obviously normalized, 
  so in fact the columns are mutually orthonormal - thus we have
  \begin{displaymath}
    \mathbf{Q}^T\mathbf{Q} = \mathbf{I}
  \end{displaymath}
  as desired
\end{proof}

The theorems above establish that the algorithm produces a factorization of $\mathbf{A}$ and that the resulting matrix $\mathbf{Q}$ has
good conditioning properties. Everything so far assumed that $\mathbf{A}$ was square and full-rank but I demonstrate below that 
these restrictions may easily be removed without changing the theorems.

\subsection{Extending to rank-deficient case and rectangular \texorpdfstring{$\mathbf{A}$}{A}}
\label{sec:lowrank}
One can use the algorithm \ref{alg:genqr} without significant modification on rank-deficient matrices and rectangular matrices.
For this we need a "breakdown condition" on the normalization value $\gamma_j$ in equation \ref{eq:recurrence3}. When
$\gamma_j \approx 0$ that means the minimization problem has found a nearly exact answer meaning the input matrix $\mathbf{A}$
is rank deficient. To handle this case the algorithm fills the corresponding values of $\mathbf{R}$ (resulting in a $0$ on the
diagonal) but does not include the new column of $\mathbf{Q}$ corresponding to the breakdown 
and then proceeds to the next column of $\mathbf{A}$ until all columns have been processed. However many columns of $\mathbf{A}$ get
"skipped" in this fashion reduces the number of columns of $\mathbf{Q}$ and rows of $\mathbf{R}$. In other words
if the input matrix $\mathbf{A}\in\mathbb{R}^{m \times m}$ has rank $k\leq m$ then the above modifications output
$\mathbf{Q}\in \mathbb{R}^{m \times k }, \mathbf{R} \in \mathbb{R}^{k \times m }$ (similar to a "thin QR"). A factorization
in this way may readily be shown to also satisfy all of the theorems that assumes full-rank $\mathbf{A}$.

For rectangular matrices we may use the above observation and simply input the rectangular matrix into a square matrix that is zero-padded. 
The resulting matrix will be rank deficient and the earlier modifications to the algorithm will correctly produce a factorization. 
In practice one should simply use the rectangular matrices directly - but I make this observation for the purpose of extending the 
theorems proven for the square matrix case.

\subsection{Rank-revealing factorizations} 
\label{sec:rankreveal}
Following observations in \ref{sec:lowrank} we could further extend this algorithm into a "rank revealing" algorithm which
also outputs a column permutation for $\mathbf{A}$ which guarantees that the diagonal of 
$\mathbf{R}$ is \emph{decreasing}. I have done this in a pre-print \cite{atcheson2019rank} and there have proven that the resulting factorization
has the expected rank-revealing properties and can even be used as a way to compute low-rank approximations to an input matrix similar to 
classical rank-revealing QR. I found the resulting conditioning theorems, specifically \ref{thm:inverse}, very difficult to prove however and in this manuscript
sought to remove any extraneous details not relevant to this bound.

\section{Numerical Experiments}
\label{sec:experiments}

Below I provide numerical experiments to confirm the theorems conerning $\mathbf{Q}$ and to provide some intuition I also suggest a way to interpret
$\mathbf{Q}$ as a basis - similarly to how it is interpreted for classic QR. I do these studies for both the $l^1$ and $l^{\infty}$ norms.
I describe in the appendix section \ref{app:implement} how the factorization was implemented and provide example code for this purpose.

\subsection{Numerical studies confirming bounds on \texorpdfstring{$\mathbf{Q}$}{Q}}
\label{sec:numericbound}

For these studies I seek to confirm that the forward bound \ref{thm:forward} and the inverse bound \ref{thm:inverse} are indeed
independent of any input $\mathbf{A}$. I then attempt to quantify the dependence of these bounds on $m$ as proof of theorem \ref{thm:inverse}
resulted in exponentially decaying bound as $m\to \infty$, resulting in exponentially growing inverse matrix norm. Since this theorem
was proved using an \emph{arbitrary} norm it stands to reason that \emph{specific concrete} norms could improve on this growth
significantly. To show these I randomly sample matrices with different condition numbers and sizes $m$, apply
algorithm \ref{alg:genqr}, and then compute the forward and inverse bounds as matrix norms. I do this first for the $l^1$ 
case and then follow with the $l^{\infty}$ case

\begin{figure}[H]
    \centering
    \begin{subfigure}[b]{0.4\textwidth}
        \includegraphics[width=\textwidth]{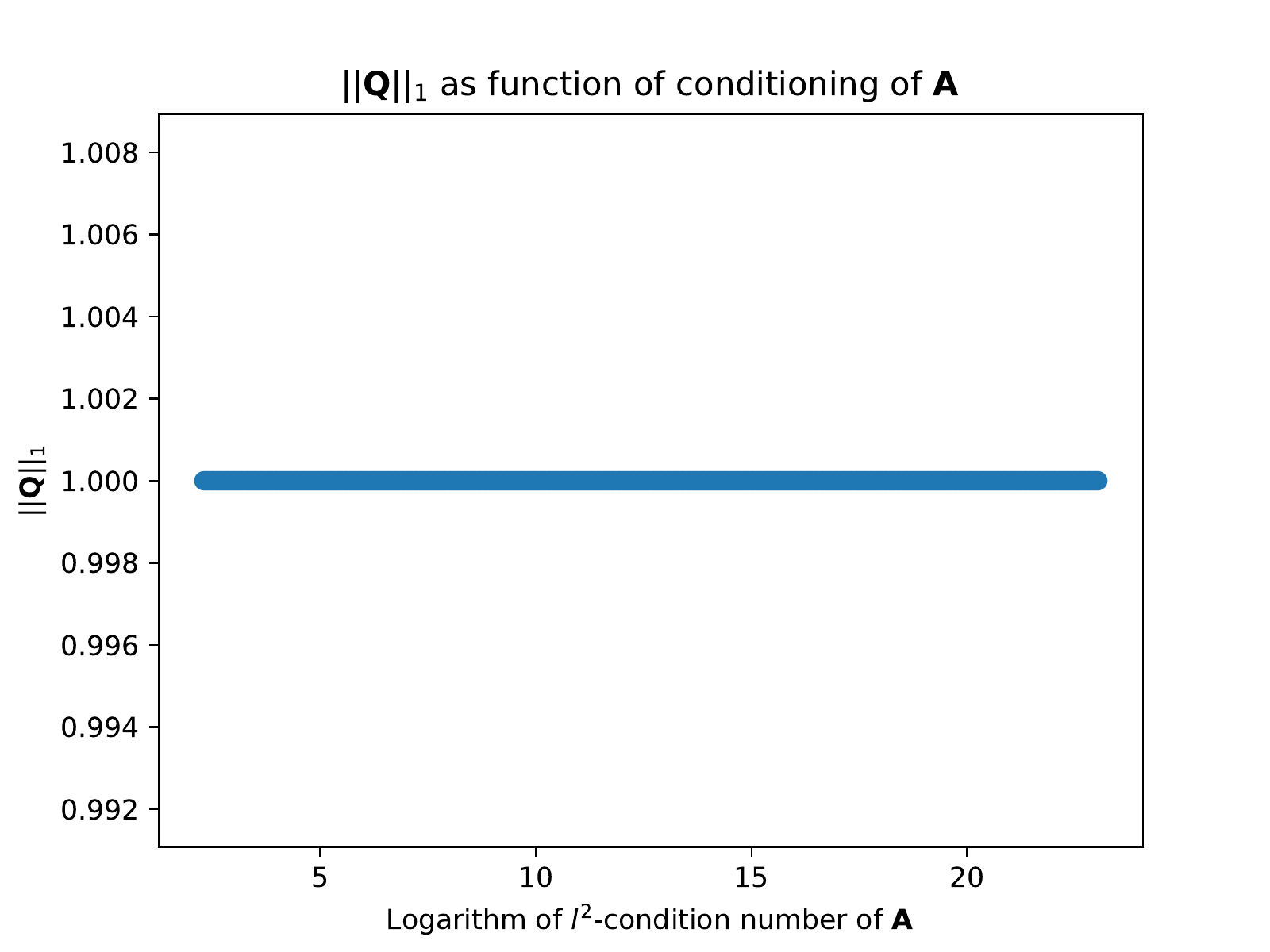}
    \end{subfigure}
    ~ 
    \begin{subfigure}[b]{0.4\textwidth}
        \includegraphics[width=\textwidth]{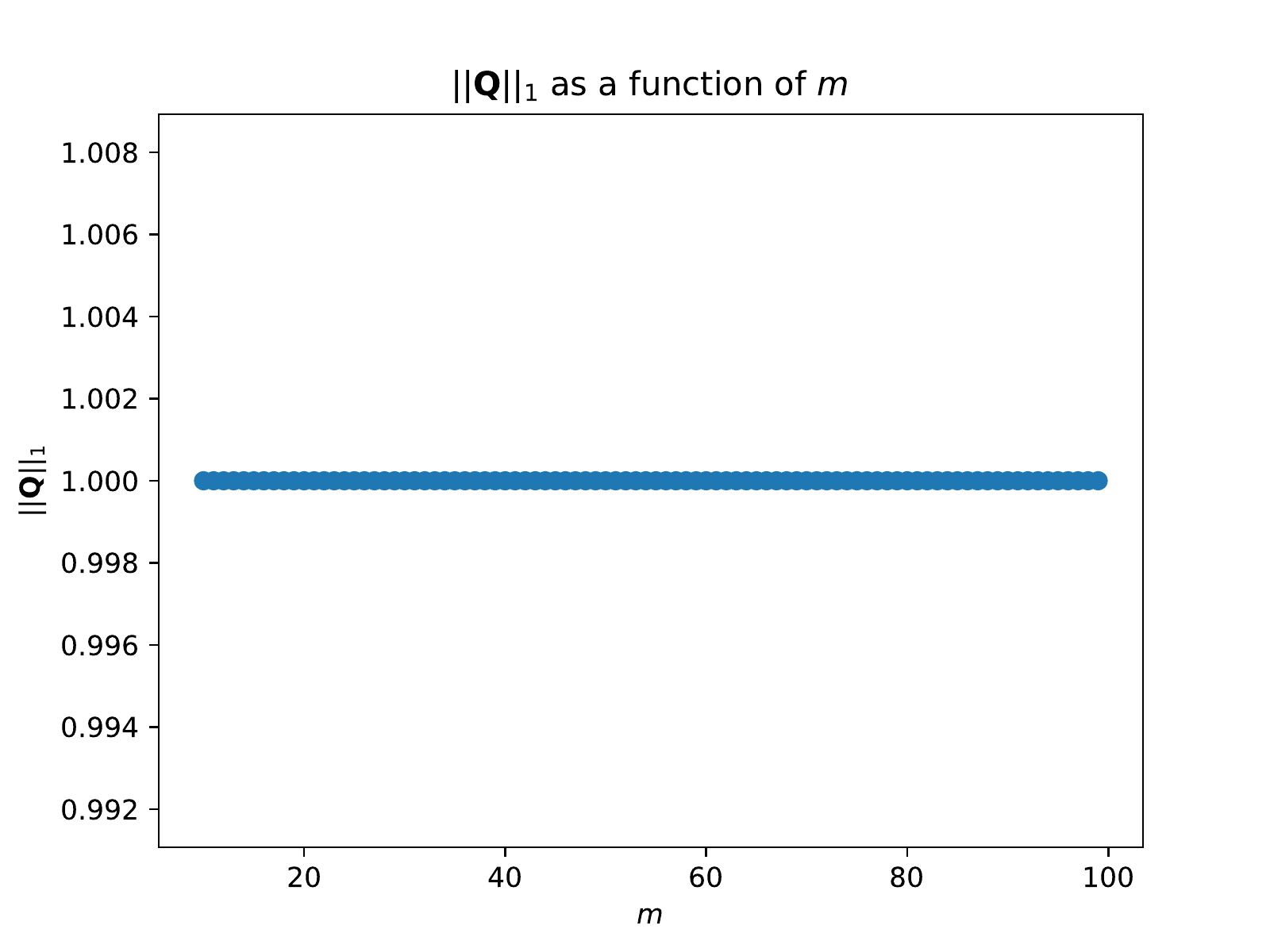}
    \end{subfigure}

    \begin{subfigure}[b]{0.4\textwidth}
        \includegraphics[width=\textwidth]{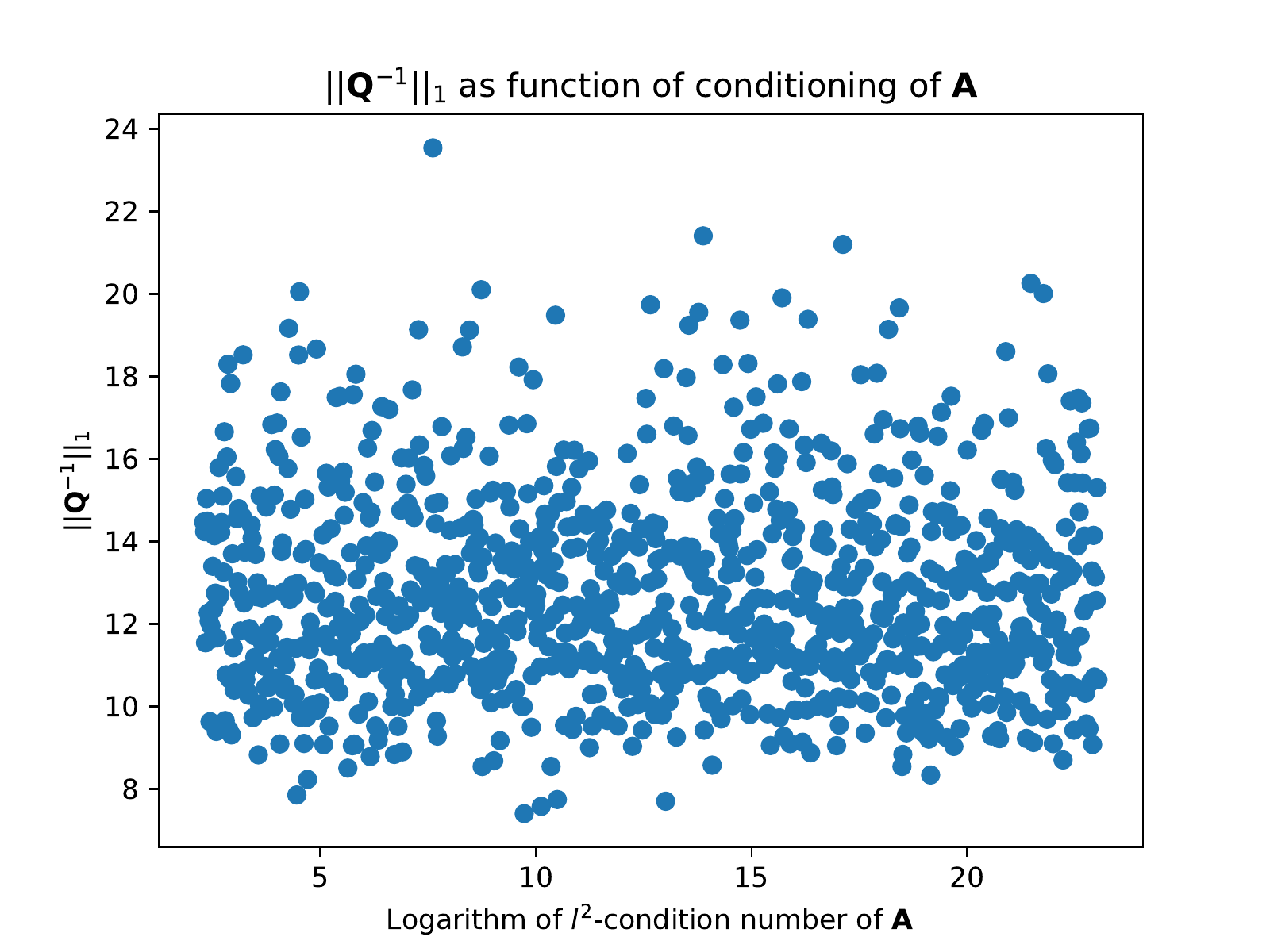}
    \end{subfigure}
    ~ 
    \begin{subfigure}[b]{0.4\textwidth}
        \includegraphics[width=\textwidth]{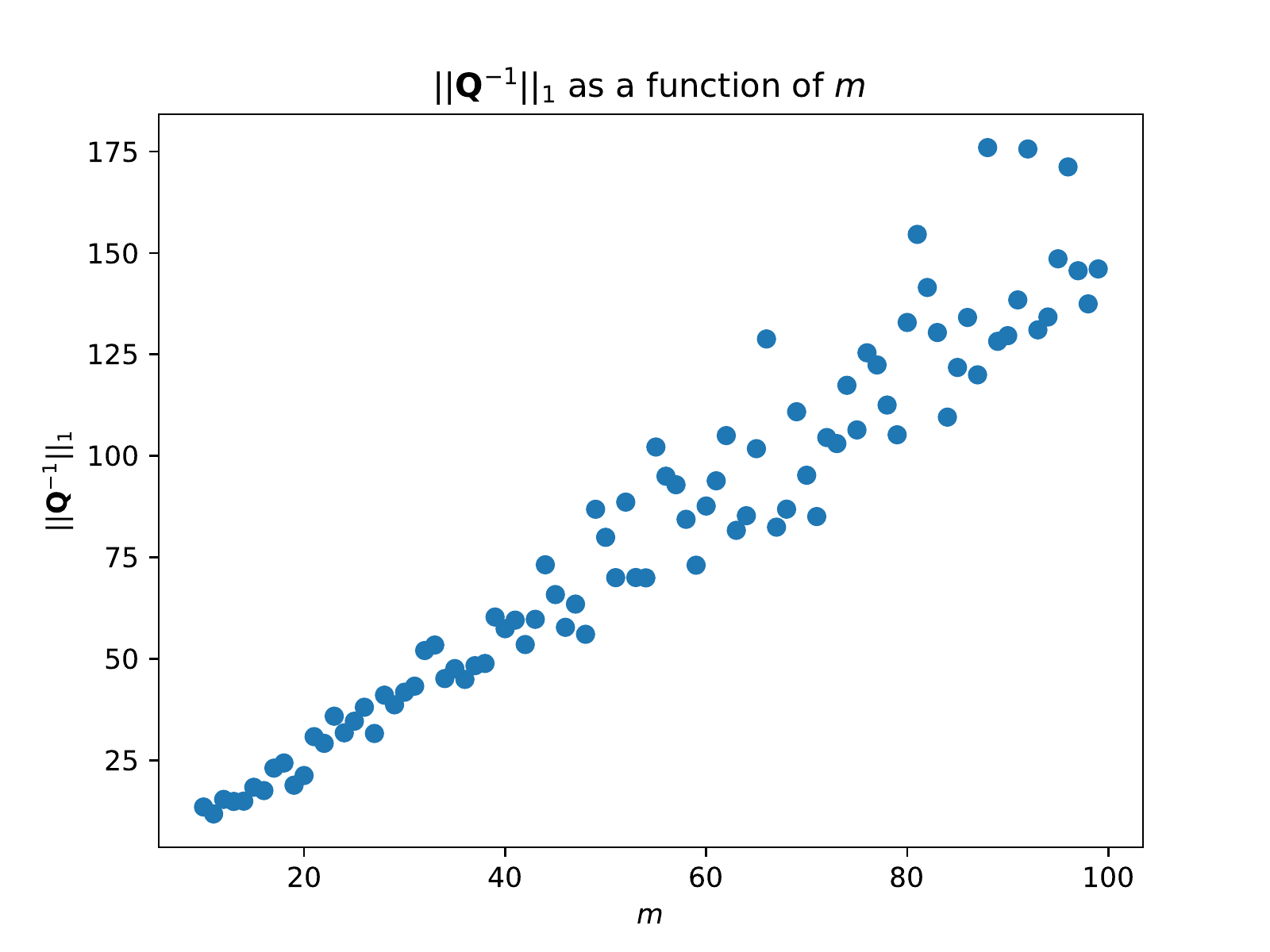}
    \end{subfigure}

    \caption{Using $l^1$ norm: Dependence of forward and inverse bounds of $\mathbf{Q}$ on $m$ and $\kappa \left ( \mathbf{A} \right )$}
\end{figure}
and also the $l^{\infty}$ case below
\begin{figure}[H]
    \centering
    \begin{subfigure}[b]{0.4\textwidth}
        \includegraphics[width=\textwidth]{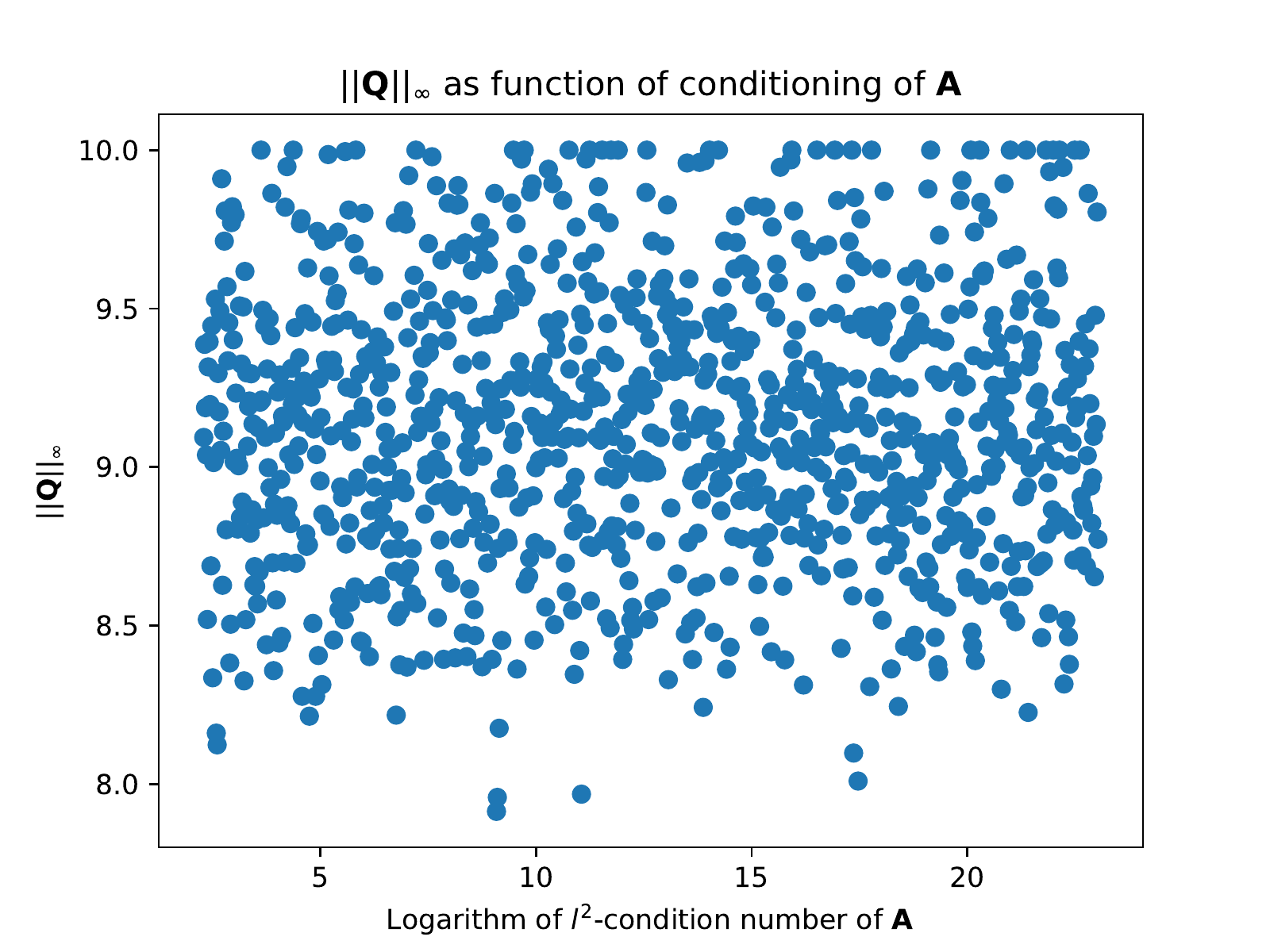}
    \end{subfigure}
    ~ 
    \begin{subfigure}[b]{0.4\textwidth}
        \includegraphics[width=\textwidth]{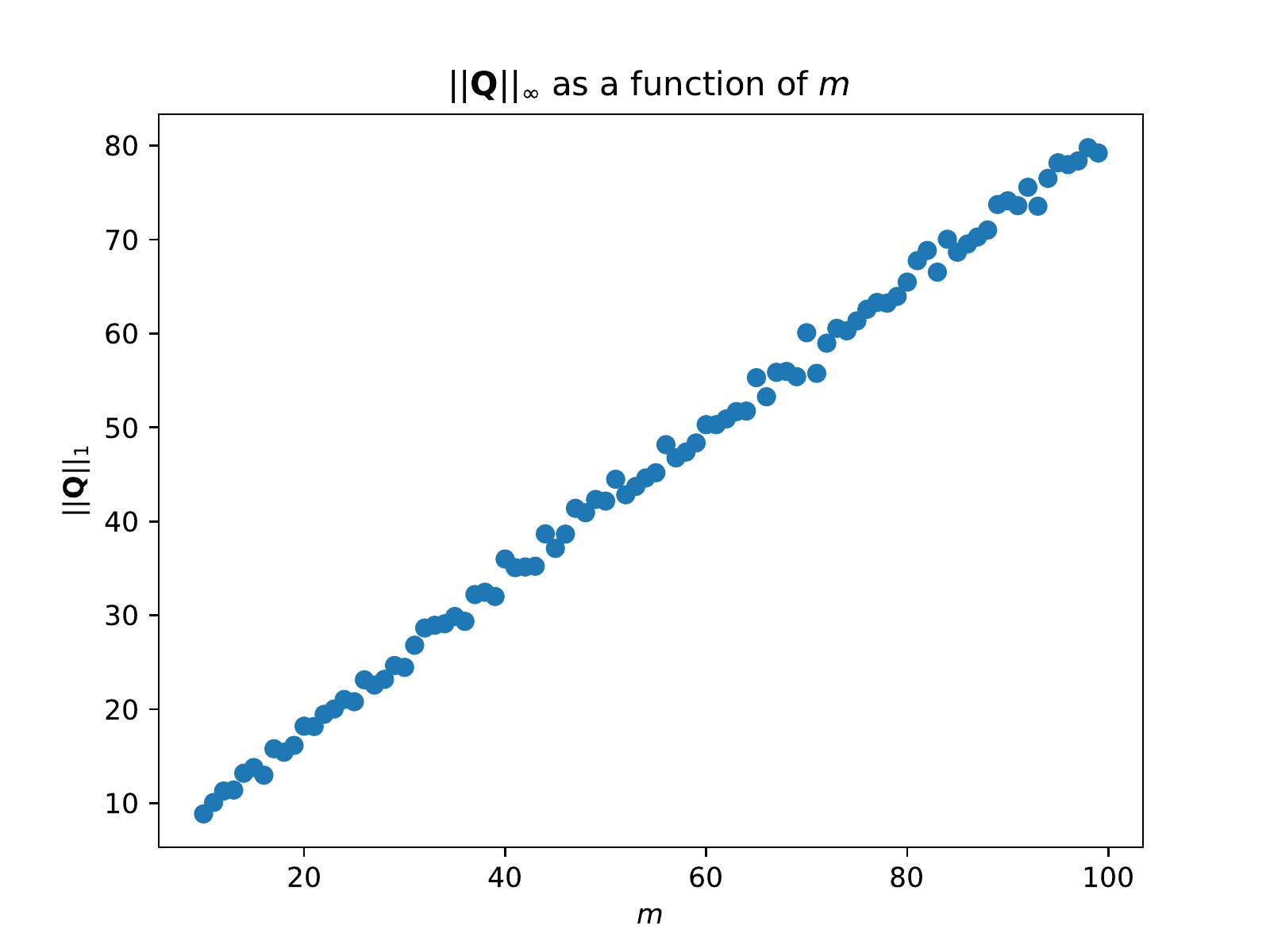}
    \end{subfigure}

    \begin{subfigure}[b]{0.4\textwidth}
        \includegraphics[width=\textwidth]{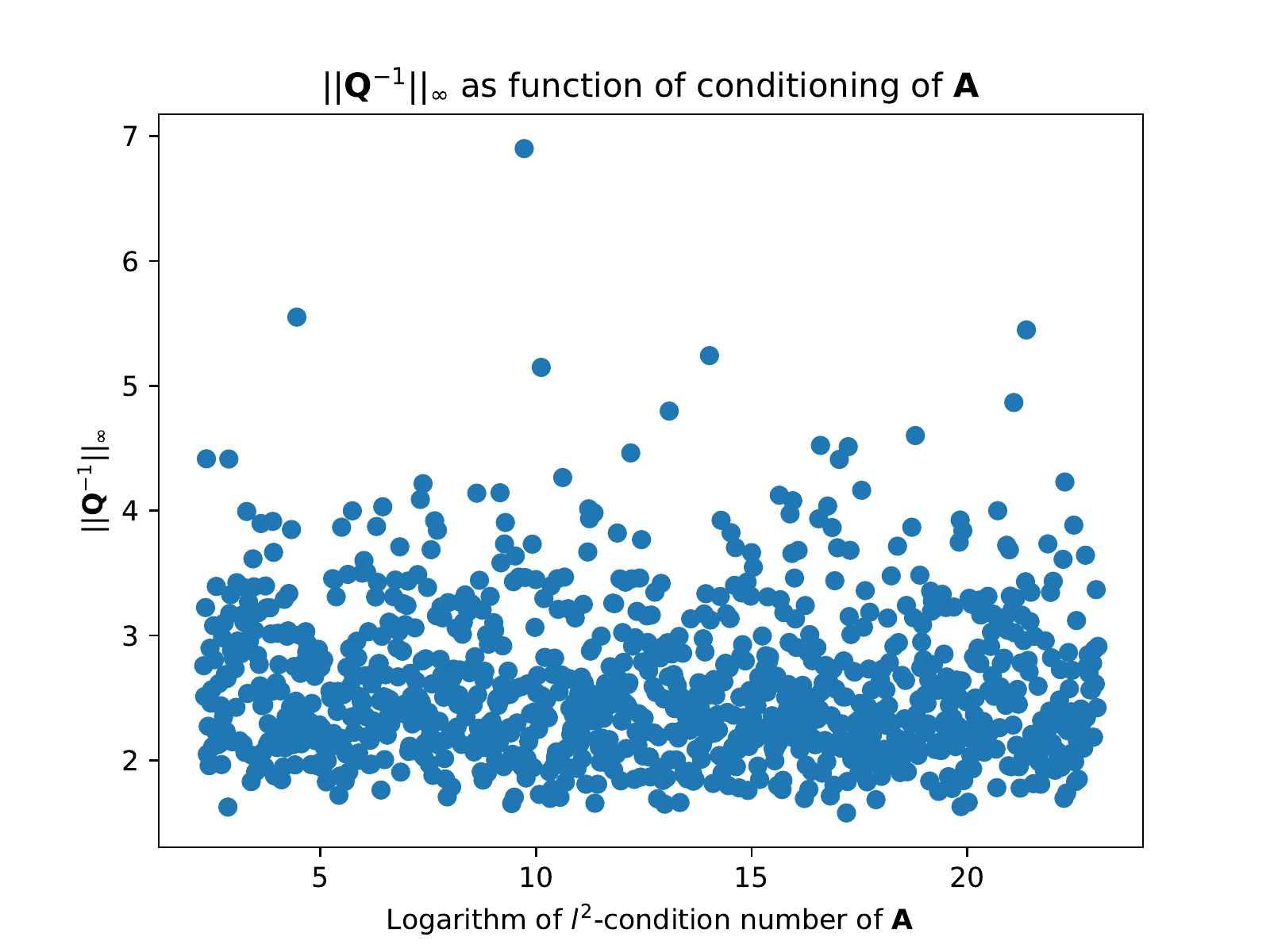}
    \end{subfigure}
    ~ 
    \begin{subfigure}[b]{0.4\textwidth}
        \includegraphics[width=\textwidth]{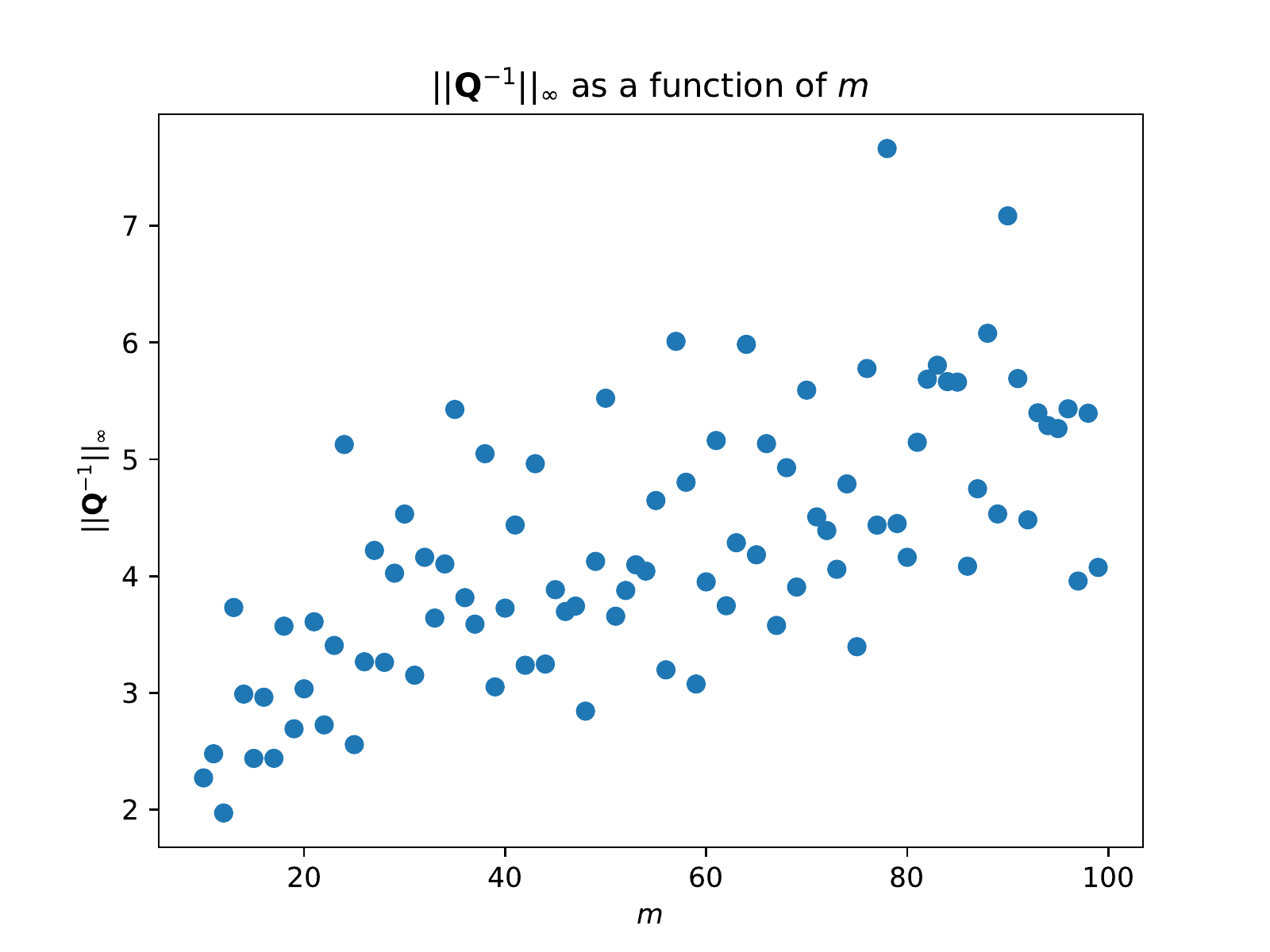}
    \end{subfigure}

    \caption{Using $l^{\infty}$ norm: Dependence of forward and inverse bounds of $\mathbf{Q}$ on $m$ and $\kappa \left ( \mathbf{A} \right )$}
\end{figure}

What we see here is confirmation that the bounds are independent of $\mathbf{A}$ and that the bounds do not grow/decay exponentially
in $m$ though there is what appears to be linear growth.

\subsection{Interpreting columns of \texorpdfstring{$\mathbf{Q}$}{Q} as a basis}
\label{sec:numericbasis}

To help provide extra intuition for the matrix $\mathbf{Q}$ I show here we may interpret it in much the same way we interpret 
this matrix when it arises from a classic QR factorization - as an optimized basis.

To illustrate this I take the Vandermonde matrix $\mathbf{V}$ defined as follows:

\begin{align*}
  m &= 400 \\
  n &= 5 \\
  h &= \frac{2}{m-1} \\
  x_i &= -1 + h*(i-1) \text{ } (i=1,\ldots,m) \\
  V_{i,j} &= x_i^{j-1} \text{ } (i=1,\ldots,m),(j=1,\ldots,n)
\end{align*}
or, in other words, the $j-th$ column of the Vandermonde matrix is the $j$-th monomial applied to a sampling of its domain, in this
case $400$ equally spaced points from the interval $[-1,1]$.
I plot the first few polynomials below for the unaltered Vandermonde, the $l^1$ $\mathbf{Q}$, and the $l^{\infty}$ $\mathbf{Q}$.

\begin{figure}[H]
    \centering
    \begin{subfigure}[b]{0.4\textwidth}
        \includegraphics[width=\textwidth]{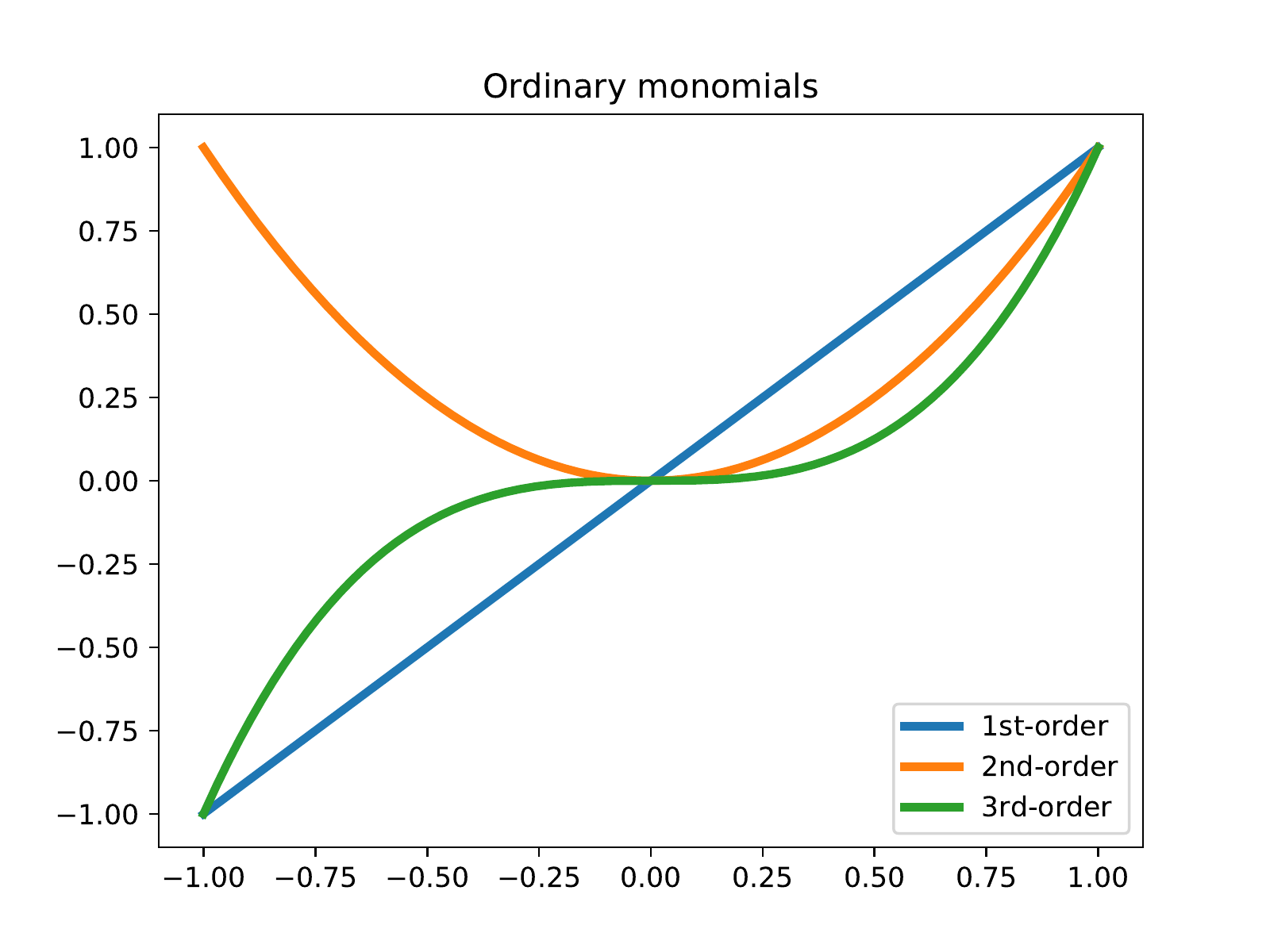}
    \end{subfigure}
    ~ 
    \begin{subfigure}[b]{0.4\textwidth}
        \includegraphics[width=\textwidth]{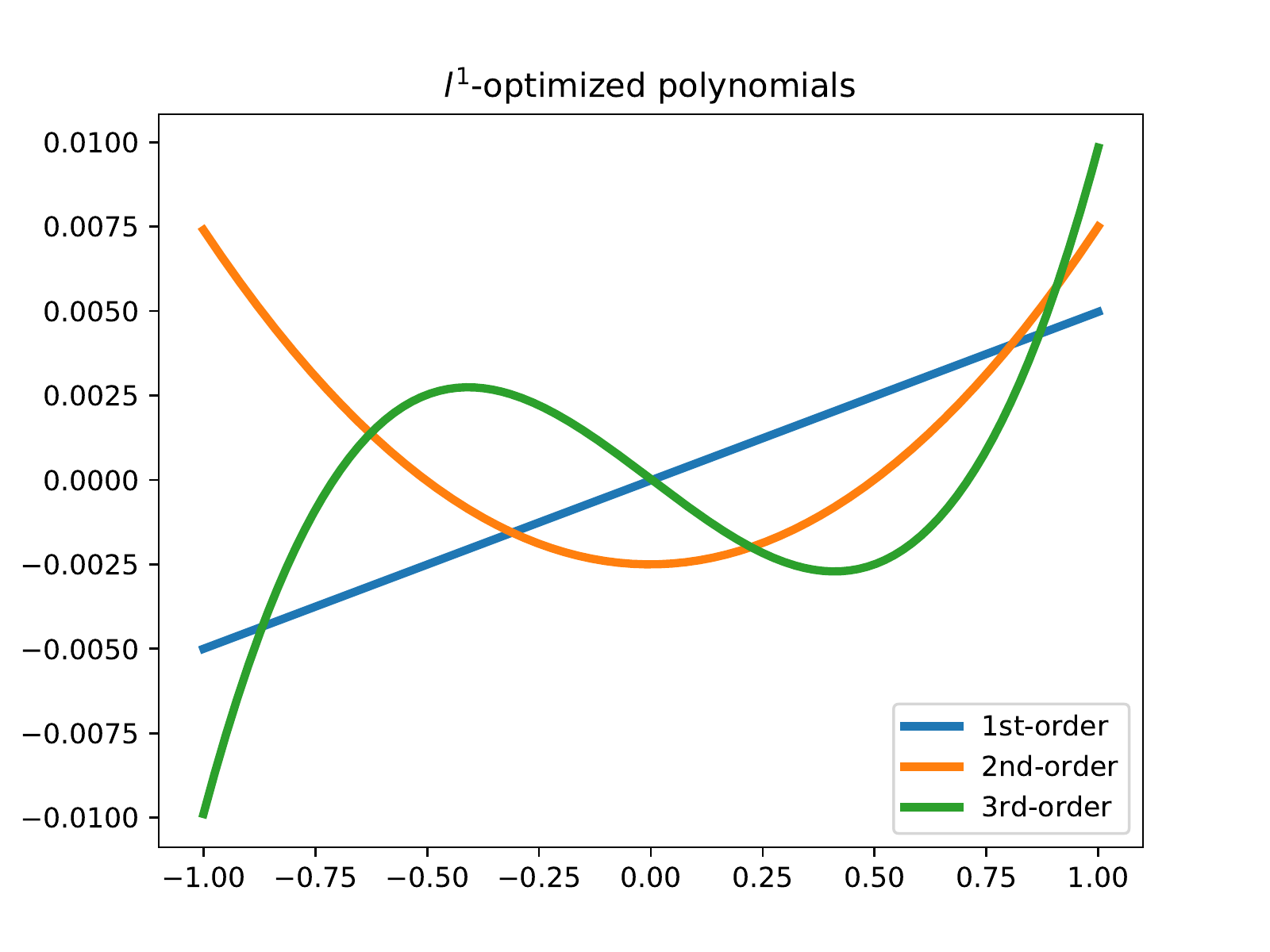}
    \end{subfigure}

    \begin{subfigure}[b]{0.4\textwidth}
      \includegraphics[width=\textwidth]{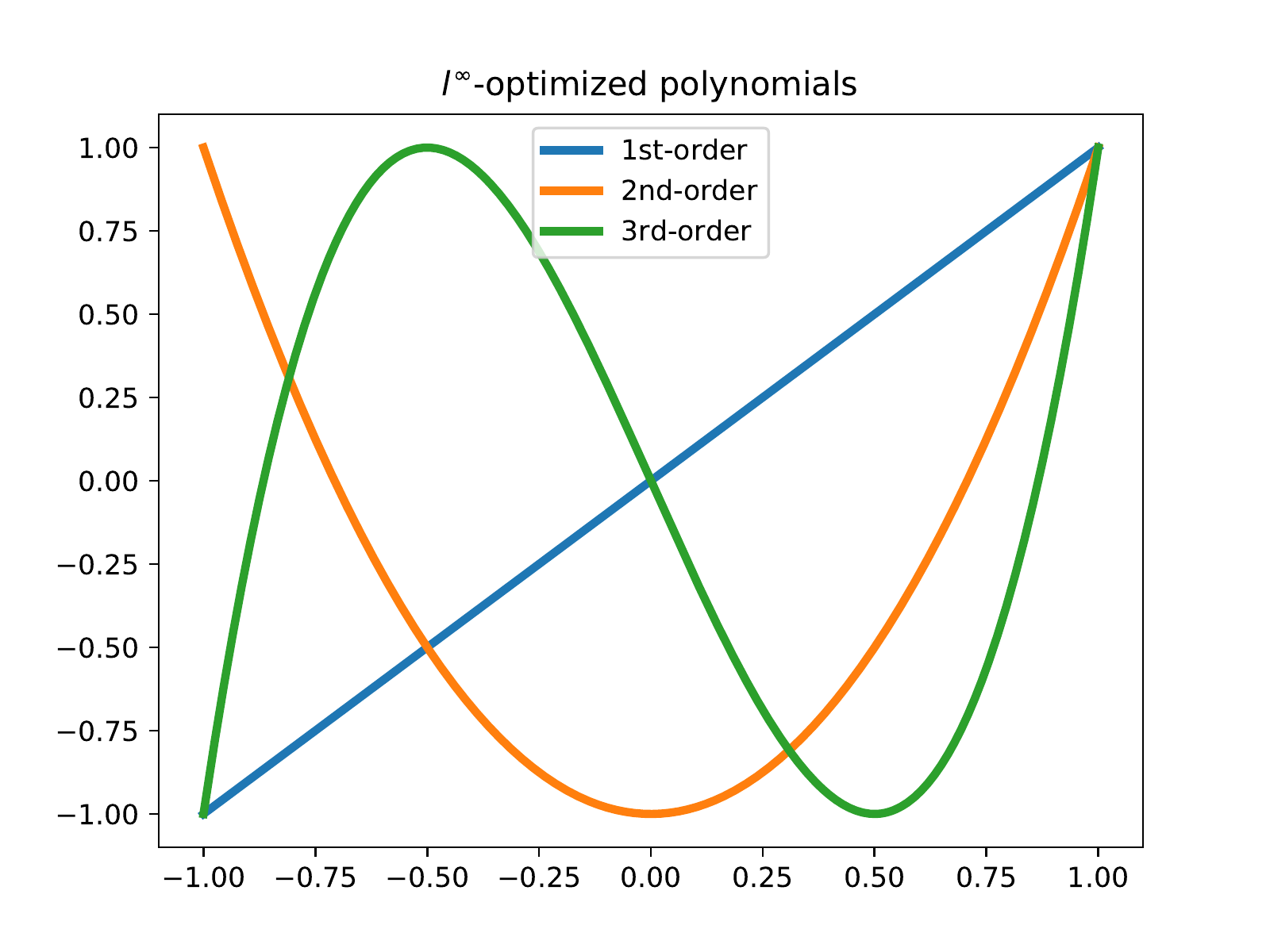}
    \end{subfigure}

    \caption{Applying algorithm \ref{alg:genqr} to Vandermonde matrix to generate different polynomial basis}
\end{figure}
Note that the $l^{\infty}$ plot effectively is an approximation to Chebyshev polynomials which would get better for larger $m$.

\section{Conclusions}
\label{sec:conclusions}

I demonstrated that algorithm \ref{alg:genqr} produces a factorization of an input matrix such that
its factors satisfy analogous properties to classical
QR factorization, but instead depend on potentially non-Euclidean norms (which are user-specified).
I showed specifically in theorem \ref{thm:forward} and \ref{thm:inverse} that the $\mathbf{Q}$ factor
is "well behaved" with respect to the input matrix $\mathbf{A}$. I illustrated the mathematical
facts with numerical experiments that validated the principle and showed we can likely improve
the constants - at least in the case of the $l^1,l^{\infty}$ norms.

Given the ongoing research of exploiting novel properties of non-Euclidean norms for matrix factorizations
I hope to use this as a step towards QR-like factorizations that can build in these properties from other 
norms.

\appendix
\section{Python implementation} 
\label{app:implement}
The key to implementing \ref{alg:genqr} is the minimization problem \ref{eq:recurrence1}. Fortunately for the
$l^1$ and $l^{\infty}$ norms we can easily formulate the minimization problems as linear programs and solve it with the
simplex algorithm - see e.g. \cite{siaml1} for the $l^1$ case and \cite{barrodale1974improved} for the $l^{\infty}$
case.

I implement these minimum-norm solvers in two different ways - one way is highly optimized and enables factorizing much larger
matrices but depends on the closed-source NAG library \cite{nagc} by way of the Python interface \cite{nagpython}. For the 
$l^1$ norm I used the NAG routine "e02gac" and for the $l^{\infty}$ norm I used the NAG routine "e02gcc". This way of computing
the factorizations is preferred because it is much more efficient. Since some may not have access to the NAG library however
I also provide a way to solve the minimization problems directly with linear programs using NumPy \cite{harris2020array} and
SciPy \cite{2020SciPy-NMeth}.

The plain SciPy implementation of the $l^1$ solver is in the soruce listing \ref{impl:l1}, the plain SciPy implementation of the $l^{\infty}$ solver
is in source listing \ref{impl:linf}. Finally the actual QR algorithm is in \ref{impl:qr}. Note that the norm \emph{and} the solver are input callbacks
so that one may use faster solvers as I have done for the NAG library variants.

\begin{figure}[H]
  \caption{Solving $l^1$ minimization problem with linear program in SciPy}
  \label{impl:l1}
\begin{python}
import numpy as np
import scipy.optimize as opt
def lst1norm(A,b):
    (m,n)=A.shape
    nvars=m+n
    ncons=2*m
    cons=np.zeros((ncons,nvars))
    cons[0:m,0:m]=-np.identity(m)
    cons[0:m,m:m+n]=A
    cons[m:2*m,0:m]=-np.identity(m)
    cons[m:2*m,m:m+n]=-A
    c=np.zeros(nvars)
    c[0:m]=1.0
    ub=np.zeros(ncons)
    ub[0:m]=b
    ub[m:2*m]=-b
    bounds=[]
    for i in range(0,m):
        bounds.append((0,None))
    for i in range(m,m+n):
        bounds.append((None,None))
    out=opt.linprog(c,cons,ub,None,None,bounds,options={'tol':1e-10,'lstsq' : True})
    return (out.x[m:m+n],out.fun)
\end{python}

\end{figure}

\begin{figure}[H]
  \caption{Solving $l^{\infty}$ minimization problem with linear program in SciPy}
  \label{impl:linf}
\begin{python}
  import numpy as np
  import scipy.optimize as opt
def lstinfnorm(A,b):
    (m,n)=A.shape
    #First n variables are x, last variable is "c" representing the inf-norm
    nvars=n+1
    ncons=2*m
    cons=np.zeros((ncons,nvars))
    ub=np.zeros(ncons)
    #First linear constraint: Ax-b<=c -->  Ax-c<=b
    cons[0:m,0:n]=A
    cons[0:m,n]=-1.0
    ub[0:m]=b
    #Second linear constraint: b-Ax<=c -->-Ax-c<=-b
    cons[m:2*m,0:n]=-A
    cons[m:2*m,n]=-1.0
    ub[m:2*m]=-b
    #Objective function: minimize c
    coeffs=np.zeros(nvars)
    coeffs[n]=1.0

    #No bounds for "x"
    bounds=[]
    for i in range(0,n):
        bounds.append((None,None))
    #But "c" should be nonnegative
    bounds.append((0,None))

    out=opt.linprog(coeffs,cons,ub,None,None,bounds,options={'tol':1e-10,'lstsq' : True})
    return (out.x[0:n],out.fun)

\end{python}

\end{figure}

\begin{figure}[H]
  \caption{Python QR algorithm}
  \label{impl:qr}
\begin{python}
import numpy as np

USE_NAG=False
#If NAG library is available use it, otherwise fall back to SciPy+linprog
#solvers
try:
    USE_NAG=True
    from naginterfaces.library.fit import glin_l1sol
    from naginterfaces.library.fit import glin_linf
except:
    from plain_scipy_solvers import lst1norm,lstinfnorm

#Simple wrapper that chooses NAG if available, otherwise uses SciPy
def l1solve(A,b):
    m,n=A.shape
    if USE_NAG:
        B=np.zeros((m+2,n+2))
        B[0:m,0:n]=A
        _,_,x,_,_,_=glin_l1sol(B,b)
        return x[0:n]
    else:
        return lst1norm(A,b)[0]

#Simple wrapper that chooses NAG if available, otherwise uses SciPy
def linfsolve(A,b):
    m,n=A.shape
    if USE_NAG:
        B=np.zeros((n+3,m+1))
        B[0:n,0:m]=A.T
        relerr=0.0
        _,_,_,x,_,_,_ = glin_linf(n,B,b,relerr)
        return x
    else:
        return lstinfnorm(A,b)[0]

#Note the callback inputs. These must be consistent. e.g.
#if the input norm is the l1-norm, then the solver must solve the
#l1-norm-minimization problem
def qr(A,norm=lambda x : np.linalg.norm(x,ord=1),solver=l1solve):
    m,n=A.shape
    Q=np.zeros((m,n))
    R=np.zeros((n,n))
    #First column of Q is normalized first column of A
    gamma=norm(A[:,0])
    Q[:,0]=A[:,0]/gamma
    #First entry of R is the normalization factor
    R[0,0]=gamma
    for i in range(1,n):
        #Find the best combination of existing Q vectors to match next column of A
        c=solver(Q[:,0:i],A[:,i])
        #calculate residual
        r=A[:,i]-Q[:,0:i]@c
        #Get normalization factor
        gamma=norm(r)
        #New Q column is normalized residual
        Q[:,i]=r/gamma
        #Upper triangular part of r are the coefficients c
        R[0:i,i]=c
        #Diagonal part is the normalization factor
        R[i,i]=gamma
    return Q,R
\end{python}

\end{figure}

\bibliographystyle{siamplain}
\bibliography{references}

\end{document}